\documentclass[a4paper]{amsart}

\usepackage{amsfonts, latexsym, amssymb, graphicx,
multicol, mathrsfs, color, amscd, verbatim, paralist,
xspace, url, euscript, stmaryrd,  amsmath, amscd, enumitem,
bbold, multirow, tikz, mathtools, mathabx}
\usepackage[all,pdf,cmtip]{xy}

\usepackage[colorlinks, linkcolor=blue, citecolor=magenta, urlcolor=cyan]{hyperref}

\usepackage{mathabx}

\def\ra{\rightarrow}

\newcommand\co{\colon} 

\newtheorem{theorem}{THEOREM}[section]

\newtheorem{proposition}[theorem]{Proposition}
\newtheorem{lemma}[theorem]{Lemma}
\newtheorem{conjecture}[theorem]{Conjecture}
\newtheorem{example}[theorem]{Example}
\theoremstyle{definition}

\theoremstyle{remark}
\newtheorem{remark}[theorem]{Remark}

\newcommand\CC{{\mathbb C}}

\newcommand\NN{{\mathbb N}}

\def\GL{\mathop{\rm GL}\nolimits}
\def\SL{\mathop{\rm SL}\nolimits}

\def\Sym{\mathop{\rm Sym}\nolimits}
\def\fm{{\mathfrak m}}
\def\Soc{\mathop{\rm Soc}\nolimits}

    \newcommand\cQ{\mathcal{Q}}

\makeatletter
\def\blfootnote{\xdef\@thefnmark{}\@footnotetext}
\makeatother

\begin{document}

\title[On the contravariant arising from hypersurface singularities]{On the contravariant of homogeneous forms
\vspace{0.1cm}\\
arising from isolated hypersurface singularities}
\author[Isaev]{Alexander Isaev}

\address{Mathematical Sciences Institute\\
Australian National University\\
Acton, ACT 2601, Australia}
\email{alexander.isaev@anu.edu.au}

\maketitle

\thispagestyle{empty}

\pagestyle{myheadings}

\begin{abstract} 
Let $\cQ_n^d$ be the vector space of homogeneous forms of degree $d\ge 3$ on $\CC^n$, with $n\ge 2$. The object of our study is the map $\Phi$, introduced in earlier articles by J. Alper, M. Eastwood and the author, that assigns to every form for which the discriminant $\Delta$ does not vanish the so-called associated form lying in the space $\cQ_n^{n(d-2)*}$. This map is a morphism from the affine variety $X_n^d:=\{f\in\cQ_n^d:\Delta(f)\ne 0\}$ to the affine space $\cQ_n^{n(d-2)*}$. Letting $p$ be the smallest integer such that the product $\Delta^p\Phi$ extends to a morphism from $\cQ_n^d$ to $\cQ_n^{n(d-2)*}$, one observes that the extended map defines a contravariant of forms in $\cQ_n^d$. In the present paper we obtain upper bounds for $p$ thus providing estimates for the contravariant's degree.
\end{abstract}

\section{Introduction}\label{intro}
\setcounter{equation}{0}

In this paper we discuss a curious connection between complex singularity theory and classical invariant theory proposed in \cite{EI1} and further explored in \cite{AI1}, \cite{AI2}, \cite{AIK}. What follows is written mainly for complex analysts and geometers, thus some of our expositional and notational choices may not be up to the taste of a reader with background in algebra, for which we apologize in advance.

Let ${\mathcal Q}_n^d$ be the vector space of homogeneous forms of degree $d$ on $\CC^n$, where $n\ge 2$, $d\ge 3$. Fix $f\in\cQ_n^d$ and consider the hypersurface $V_f:=\{z\in\CC^n\mid f(z)=0\}$. We will be interested in the situation when the singularity of $f$ at the origin is isolated, or, equivalently, when the discriminant $\Delta(f)$ of $f$ does not vanish. In this case, define $M_f:=\CC[z_1,\dots,z_n]/(f_{z_{{}_1}},\dots,f_{z_{{}_n}})$ to be the {\it Milnor algebra}\, of the singularity. Since the partial derivatives $f_{z_{{}_1}},\dots,f_{z_{{}_n}}$ have no common zeros away from the origin, the Nullstellensatz implies that the algebra $M_f$ is isomorphic to the local algebra $\CC[[z_1,\dots,z_n]]/(f_{z_{{}_1}},\dots,f_{z_{{}_n}})$. By the Mather-Yau theorem (see \cite{MY} and also \cite{Be}, \cite{Sh}, \cite[Theorem 2.26]{GLS}), the isomorphism class of $M_f$ determines the germ of the hypersurface $V_f$ at the origin up to biholomorphism, hence the form $f$ up to linear equivalence.

In fact, for a general isolated hypersurface singularity defined by (the germ of) a holomorphic function $F$, the Mather-Yau theorem states that, remarkably, the singularity is determined, up to biholomorphism, by $n$ and the isomorphism class of the {\it Tjurina algebra}\, $T_F:=\CC[[z_1,\dots,z_n]]/(F,F_{z_{{}_1}},\dots,F_{z_{{}_n}})$. The proof of the Mather-Yau theorem is not constructive, and it is an important open problem---called {\it the reconstruction problem}---to understand explicitly how the singularity is encoded in the corresponding Tjurina algebra. In this paper we concentrate on the {\it homogeneous}\, case as set out in the previous paragraph (notice that $T_f=M_f$). In this situation, the reconstruction problem was solved in \cite{IK}, where we proposed a simple algorithm for extracting the linear equivalence class of the form $f$ from the isomorphism class of $M_f$. An alternative (invariant-theoretic) approach to the reconstruction problem---which applies to the more general class of {\it quasihomogeneous}\, isolated hypersurface singularities---was developed in article \cite{EI1}, where we proposed a method for extracting certain numerical invariants of the singularity directly from its Milnor algebra (see also \cite{EI2}, \cite{I}). Already in the case of homogeneous singularities this approach leads to a curious concept that deserves attention regardless of the reconstruction problem and that is interesting from the purely classical invariant theory viewpoint. This concept is the focus of the present paper.

We will now briefly describe the idea behind it with details postponed until Section \ref{setup}. Let $\fm$ be the (unique) maximal ideal of $M_f$ and $\Soc(M_f)$ the socle of $M_f$, defined as $\Soc(M_f):=\{x\in M_f: x\,\fm=0\}$. It turns out that $M_f$ is a Gorenstein algebra, i.e., $\dim_{\CC}\Soc(M_f)=1$, and, moreover, that $\Soc(M_f)$ is spanned by the image $\widehat{H(f)}$ of the Hessian $H(f)$ of $f$ in $M_f$. Observing that $H(f)$ is an element of $\cQ_n^{n(d-2)}$, one can then introduce a form defined on the $n$-dimensional quotient $\fm/\fm^2$ with values in $\Soc(M_f)$ as follows:
$$
\fm/\fm^2 	 \to \Soc(M_f), \quad x  \mapsto y^{\,n(d-2)},
$$
where $y$ is any element of ${\mathfrak m}$ that projects to $x\in{\mathfrak m}/{\mathfrak m}^2$.  There is a canonical isomorphism ${\mathfrak m}/{\mathfrak m}^2\cong \CC^{n*}$ and, since $\widehat{H(f)}$ spans the socle, there is also a canonical isomorphism $\Soc(M_f) \cong \CC$. Hence, one obtains a form ${\mathbf f}$ of degree $n(d-2)$ on  $\CC^{n*}$ (which can be interpreted as an element of the dual space $\cQ_n^{n(d-2)*}$), called the {\it associated form}\, of $f$.

The principal object of our study is the morphism
$$
\Phi:X_n^d\to \cQ_n^{n(d-2)*},\quad f\mapsto{\mathbf f}
$$
of affine algebraic varieties, where $X_n^d$ is the variety of forms in $\cQ_n^d$ with nonzero discriminant. This map has a $\GL_n$-equivariance property (see Proposition \ref{equivariance}), and one of the reasons for our interest in $\Phi$ is the following intriguing conjecture proposed in \cite{AI1}, \cite{EI1}:

\begin{conjecture}\label{conj2} For every regular $\GL_n$-invariant function $S$ on $X_n^d$ there exists a rational $\GL_n$-invariant function $R$ on $\cQ_n^{n(d-2)*}$ defined at all points of the set\, $\Phi(X_n^d)\subset {\mathcal Q}_n^{n(d-2)*}$ such that $R\circ\Phi=S$.
\end{conjecture}

Observe that, if settled, Conjecture \ref{conj2} would imply an invariant-theoretic solution to the reconstruction problem in the homogeneous case. Indeed, on the one hand, regular $\GL_n$-invariant functions on $X_n^d$ are easily seen to separate the $\GL_n$-orbits (see, e.g., \cite[Proposition 3.1]{EI1}), and, on the other hand, the result of the evaluation of any rational $\GL_n$-invariant function on the associated form ${\mathbf f}$ depends only on the isomorphism class of $M_f$. Thus, the conjecture would yield a complete system of biholomorphic invariants of homogeneous isolated hypersurface singularities constructed from the algebra $M_f$ alone. So far, Conjecture \ref{conj2} has been confirmed for binary forms (see \cite{EI1}, \cite{AI2}), and its weaker variant (which does not require that the function $R$ be defined on the entire image of $\Phi$) has been established for all $n$ and $d$ (see \cite{AI1}).

The conjecture is also rather interesting from the purely invariant-theoretic point of view. Indeed, if settled, it would imply that the invariant theory of forms in ${\mathcal Q}_n^d$ can be extracted, by way of the morphism $\Phi$, from that of forms in ${\mathcal Q}_n^{n(d-2)*}$ at least at the level of rational invariant functions, or {\it absolute invariants}. 

The main goals of the present paper are twofold. Firstly, we would like to draw the attention of the complex-analytic audience to the concept of the associated form and the curious connection between complex geometry and invariant theory manifested through it as explained above and in more detail in Section \ref{setup}. Secondly, we look at another aspect of associated forms, which is not directly related to Conjecture \ref{conj2} but nevertheless contributes to classical invariant theory. Namely, letting $p$ be the smallest integer such that the product $\Delta^p\Phi$ extends to a morphism from $\cQ_n^d$ to $\cQ_n^{n(d-2)*}$, by utilizing the equivariance property of $\Phi$ one observes that this product defines a {\it contravariant}\, of degree $np(d-1)^{n-1}-n$ of forms in $\cQ_n^d$ (see Section \ref{S:contravariant} for details). It appears that for general $n$ and $d$ this contravariant is new and had not been studied prior to our recent work \cite{AIK} where it was first introduced and related to known contravariants in the cases of binary quartics, binary quintics and ternary cubics. 

One of the first natural questions arising while investigating $\Delta^p\Phi$ is that of determining or at least estimating the integer $p$ and hence the contravariant's degree. In \cite{AIK} we showed that $p=1$ for $n=2,3$. In the present paper we find an upper bound for $p$ that works for all $n$ (see Theorem \ref{main}). If $n=2,3$, this bounds yields $p=1$ as expected while in general it is not sharp as evidenced by our more precise estimates obtained for $n=4, 5$ in Propositions \ref{n=4}, \ref{n=5}. 

{\bf Acknowledgements.} This work is supported by the Australian Research Council.

\section{Preliminaries on associated forms}\label{setup}
\setcounter{equation}{0}

In this section we provide an introduction to associated forms and their properties that requires only a minimal background in algebra, for which we include detailed references.

Let ${\mathcal Q}_n^d := \Sym^d (\CC^{n*})$ be the vector space of homogeneous forms of degree $d$ on $\CC^n$ where $n\ge 2$. Its dimension is given by the well-known formula
\begin{equation}
\dim_{\CC}{\mathcal Q}_n^d=\left(
\begin{array}{c}
d+n-1\\
d
\end{array}
\right).\label{dimform}
\end{equation}
The standard action of $\GL_n:=\GL_n(\CC)$ on $\CC^n$ induces an action on ${\mathcal Q}_n^d$ as follows:
$$
(Cf)(z):=f\left(z\cdot C^{-t}\right)
$$
for $C\in\GL_n$, $f\in{\mathcal Q}_n^d$ and $z=(z_1,\dots,z_n)\in\CC^n$. Two forms that lie in the same $\GL_n$-orbit are called linearly equivalent. 

To every nonzero $f\in{\mathcal Q}_n^d$ we associate the hypersurface
$$
V_f:=\{z\in\CC^n\mid f(z)=0\}
$$
and consider it as a complex space with the structure sheaf induced by $f$. The singular set of $V_f$ is then the critical set of $f$. In particular, if $d\ge 2$ the hypersurface $V_f$ has a singularity at the origin. We are interested in the situation when this singularity is isolated, or, equivalently, when $V_f$ is smooth away from 0. This occurs if and only if $f$ is nondegenerate, i.e., $\Delta(f)\ne 0$, where $\Delta$ is the discriminant (see \cite[Chapter 13]{GKZ}). 

For $d\ge 3$ define
$$
X^d_n:=\{f\in{\mathcal Q}_n^d \mid \Delta(f)\ne 0\}.
$$
Observe that $\GL_n$ acts on the affine variety $X_n^d$ and note that every $f\in X_n^d$ is stable with respect to this action, i.e., the orbit of $f$ is closed in $X_n^d$ and has dimension $n^2$  (see, e.g., \cite[Corollary 5.24]{Mu}). It then follows by standard Geometric Invariant Theory arguments (see, e.g., \cite[Proposition 3.1]{EI1}) that regular invariant functions on $X_n^d$ separate the $\GL_n$-orbits. As explained in the introduction, this is one of the facts that links Conjecture \ref{conj2} with the reconstruction problem arising from the Mather-Yau theorem.  

Fix $f\in X^d_n$ and consider the Milnor algebra of the singularity\ of $V_f$, which is the complex local algebra
$$
M_f:=\CC[[z_1,\dots,z_n]]/(f_1,\dots,f_n),
$$
where $\CC[[z_1,\dots,z_n]]$ is the algebra of formal power series in $z_1,\dots,z_n$ with complex coefficients and $f_j:=\partial f/\partial z_j$, $j=1,\dots,n$. Since the singularity of $V_f$ is isolated, the algebra $M_f$ is Artinian, i.e., $\dim_{\CC}M_f<\infty$ (see \cite[Proposition 1.70]{GLS}). Therefore, $f_1,\dots,f_n$ is a system of parameters in $\CC[[z_1,\dots,z_n]]$, and since $\CC[[z_1,\dots,z_n]]$ is a regular local ring, $f_1,\dots,f_n$ is a regular sequence in $\CC[[z_1,\dots,z_n]]$. This yields that $M_f$ is a complete intersection (see \cite[\S 21]{Ma}).

It is convenient to utilize another realization of the Milnor algebra. Namely, by the Nullstellensatz, the algebra $\CC[z_1,\dots,z_n]/(f_1,\dots,f_n)$ is local, and it is easy to see that it is isomorphic to $M_f$. Thus, we can write  
$$
M_f=\CC[z_1,\dots,z_n]/(f_1,\dots,f_n).
$$
Let ${\mathfrak m}$ denote the maximal ideal of $M_f$, which consists of all elements represented by polynomials in $\CC[z_1,\dots,z_n]$ vanishing at the origin. By Nakayama's lemma, the maximal ideal is nilpotent and we let $\nu:=\max\{\eta\in\NN\mid {\mathfrak m}^{\eta}\ne 0\}$ be the socle degree of $M_f$.

Since $M_f$ is a complete intersection, by \cite{Ba} it is a Gorenstein algebra. This means that the socle of $M_f$, defined as
$$
\Soc(M_f):=\{x\in M_f \mid x\,{\mathfrak m}=0\},
$$
is a one-dimensional vector space over $\CC$  (see, e.g., \cite[Theorem 5.3]{H}). We then have $\Soc(M_f)={\mathfrak m}^{\nu}$. Furthermore, $\Soc(M_f)$ is spanned by the element $\widehat{H(f)}$ of $M_f$ represented by the Hessian $H(f)$ of $f$ (see, e.g., \cite[Lemma 3.3]{Sa}). Since $H(f)$ is a form in ${\mathcal Q}_n^{n(d-2)}$, it follows that $\nu=n(d-2)$. Thus, the subspace 
\begin{equation}
W_f:={\mathcal Q}_n^{n(d-2)-(d-1)}f_1+\dots+{\mathcal Q}_n^{n(d-2)-(d-1)}f_n\subset{\mathcal Q}_n^{n(d-2)}\label{subspace}
\end{equation}
has codimension 1, with the line spanned by $H(f)$ being complementary to it.

Let $e_1^*=z_1,\dots,e_n^*=z_n$ be the basis in $\CC^{n*}$ dual to the standard basis $e_1,\dots,e_n$ in $\CC^n$ and $z_1^*,\dots,z_n^*$ the coordinates of a vector $z^*\in\CC^{n*}$ (we slightly abuse notation by writing $z^*=(z_1^*,\dots,z_n^*)$). Denote by $\omega \co \Soc(M_f)\ra\CC$ the linear isomorphism given by the condition $\omega(\widehat{H(f)})=1$. Define a form ${\mathbf f}$ of degree $n(d-2)$ on $\CC^{n*}$ (i.e., an element of $\Sym^{n(d-2)} (\CC^n)$) by the formula
$$
{\mathbf f}(z^*):=\omega\left((z_1^*\widehat{z}_1+\dots+z_n^*\widehat{z}_n)^{n(d-2)}\right),\label{assocformdef}
$$
where $\widehat{z}_j$ is the element of the algebra $M_f$ represented by the coordinate function $z_j\in\CC[z_1,\dots,z_n]$. We call ${\mathbf f}$ the {\it associated form}\, of $f$.

\begin{example}\label{E:example1} \rm If $f = a_1 z_1^d + \cdots + a_n z_n^d$ for nonzero $a_i \in \CC$, then one computes $H§ (f) = (a_1 \cdots a_n)(d(d-1))^n (z_1 \cdots z_n)^{d-2}$ and
$$
{\mathbf f}(z^*) = \frac{1}{a_1 \cdots a_n} \frac{(n(d-2))!}{(d!)^n} (z_1^* \cdots z_n^*)^{d-2}.
$$
\end{example}

To calculate ${\mathbf f}$ in general, notice that if  $i_1,\dots,i_n$ are nonnegative integers such that $i_1+\dots+i_n=n(d-2)$, the product $\widehat{z}_1^{i_1}\cdots \widehat{z}_n^{i_n}$ lies in $\Soc(M_f)$, hence we have 
\begin{equation}
\widehat{z}_1^{i_1}\cdots \widehat{z}_n^{i_n}=\mu_{i_1,\dots,i_n}(f) \widehat{H(f)}\label{assocformexpppp}
\end{equation}
for some $\mu_{i_1,\dots,i_n}(f)\in\CC$. In terms of the coefficients $\mu_{i_1,\dots,i_n}(f)$ the form ${\mathbf f}$ is written as 
\begin{equation}
{\mathbf f}(z^*)=\sum_{i_1+\cdots+i_n=n(d-2)}\frac{(n(d-2))!}{i_1!\cdots i_n!}\mu_{i_1,\dots,i_n}(f)
z_1^{* i_1}\cdots z_n^{* i_n}.\label{assocformexpp}
\end{equation}

We will now show:

\begin{proposition}\label{regular} 
Each $\mu_{i_1,\dots,i_n}$ is a regular function on the affine variety $X_n^d$.
\end{proposition}

\begin{proof} Recall that for $f\in X_n^d$ the subspace $W_f$ introduced in (\ref{subspace}) has codimension 1, and the line spanned by the Hessian $H(f)$ is complementary to it in the vector space $\cQ_n^{n(d-2)}$. Let $K:=\dim_{\CC}{\mathcal Q}_n^{n(d-2)-(d-1)}$ and ${\tt m}_1,\dots,{\tt m}_K$ be the standard monomial basis in ${\mathcal Q}_n^{n(d-2)-(d-1)}$. Then $W_f$ is spanned by the products $f_r\,{\tt m}_s$, $r=1,\dots,n$, $j=s,\dots,K$. Choose a basis ${\mathbf e}_k(f):= f_{r_{{}_k}}\,{\tt m}_{s_{{}_k}}$ in $W_f$, with $k=1,\dots,N-1$, where $N:=\dim_{\CC}\cQ_n^{n(d-2)}$. We note that the indices $r_k$, $s_k$ can be assumed to be independent of the form $f$ if it varies in some Zariski open subset $U$ of $X_n^d$, and from now on we assume that this is the case. Then for every $g\in\cQ_n^{n(d-2)}$ there are $\alpha_k(f,g)\in\CC$, with $k=1,\dots,N-1$, and $\gamma(f,g)\in\CC$ such that
\begin{equation}
\alpha_1(f,g){\mathbf e}_1(f)+\dots+\alpha_{N-1}(f,g){\mathbf e}_{N-1}(f)+\gamma(f,g) H(f)=g.\label{keyeq}
\end{equation}
Notice that $\gamma(f,z_1^{i_1}\dots z_n^{i_n})=\mu_{i_1,\dots,i_n}(f)$ for $i_1+\dots+i_n=n(d-2)$ (see (\ref{assocformexpppp})).

We now expand both sides of (\ref{keyeq}) with respect to the standard monomial basis of $\cQ_n^{n(d-2)}$. As a result, we obtain a linear system of $N$ equations with the $N$ unknowns $\alpha_k(f,g)$, $\gamma(f,g)$, $k=1,\dots,N-1$. Let $A(f)$ be the matrix of this system and $D(f):=\det A(f)$. Clearly, the entries of the matrix $A(f)$ are polynomials in the coefficients of the form $f$. Furthermore, since for every $f$ and $g$ the system has a solution, $D$ does not vanish on $U$, and the solution can be found by applying Cramer's rule. 

This implies, in particular, that on the set $U$ the function $\mu_{i_1,\dots,i_n}$ is a ratio of two polynomials in the coefficients of $f$ with nonvanishing denominator. As $X_n^d$ admits an open cover by sets with this property, it follows that $\mu_{i_1,\dots,i_n}$ is a regular function on $X_n^d$. This completes the proof.\end{proof}

\begin{remark}\label{alsomain}
Theorem \ref{main} in the next section is obtained by extending the above argument and involves a further analysis of the denominator of $\mu_{i_1,\dots,i_n}$. 
\end{remark}

By Proposition \ref{regular} we have 
\begin{equation}
\mu_{i_1,\dots,i_n}=\frac{P_{i_1,\dots,i_n}}{\Delta^{p_{i_1,\dots,i_n}}}\label{formulaformus}
\end{equation}
for some $P_{i_1,\dots,i_n}\in\CC[\cQ_n^d]$ and nonnegative integer $p_{i_1,\dots,i_n}$ (here and in what follows for any affine variety $X$ over $\CC$ we denote by $\CC[X]$ its coordinate ring, which coincides with the ring ${\mathcal O}_X(X)$ of all regular functions on $X$). To obtain representation (\ref{formulaformus}), recall that for any irreducible affine variety $X$ and a regular non-zero function $h$ on $X$, any regular function on the affine open subset $\{x\in X\mid h(x)\ne 0\}$ is the ratio of a regular function on $X$ and a non-negative power of $h$ (see, e.g., \cite[Proposition 1.40]{GW}). 

Thus, we have arrived at the morphism
$$
\Phi \co X_n^d\ra \Sym^{n(d-2)} (\CC^n) ,\quad f\mapsto {\mathbf f}
$$
of affine algebraic varieties. 

Next, recall that for any $k\ge 1$ the polar pairing between the spaces $\Sym^{k}(\CC^n)$ and $\Sym^{k} (\CC^{n*})$ is given as follows. Any $P\in\Sym^{k}(\CC^n)$ is a form of degree $k$ in $e_1,\dots,e_n$, and any  $Q\in\Sym^{k} (\CC^{n*})$ is a form of degree $k$ in $e_1^*,\dots,e_n^*$, i.e., a form of degree $k$ in the variables $z_1,\dots,z_n$. Then the polar pairing between $\Sym^{k}(\CC^n)$ and $\Sym^{k} (\CC^{n*})$ is defined by
$$
\begin{array}{l}
\Sym^{k}(\CC^n)\times\Sym^{k} (\CC^{n*})\to\CC,\\
\vspace{-0.1cm}\\
(P(e_1,\dots,e_n),Q(z_1,\dots,z_n))\mapsto P\left(\partial/\partial z_1,\dots,\partial/\partial z_n\right)(Q).
\end{array}\label{polarpairing}
$$
This pairing is nondegenerate and therefore yields a canonical identification between the spaces $\Sym^{k} (\CC^n)$ and $(\Sym^{k} (\CC^{n*}))^*=\cQ_n^{k*}$ (see, e.g., \cite[Section 1.1.1]{D} for details). Using this identification, from now on we will regard the associated form as an element of ${\mathcal Q}_n^{n(d-2)*}$ and $\Phi$ as a morphism from $X_n^d$ to ${\mathcal Q}_n^{n(d-2)*}$.

The map $\Phi$ is rather natural; in particular, \cite[Proposition 2.1]{AI1} implies an equivariance property for $\Phi$. Namely, introducing an action of $\GL_n$ on the dual space ${\mathcal Q}_n^{n(d-2)*}$ in the usual way as
$$
(Cg)(h):=g(C^{-1}h),\quad g\in \cQ_n^{n(d-2)*},\, h\in {\mathcal Q}_n^{n(d-2)},\,C\in\GL_n,
$$
we have:
  
\begin{proposition}\label{equivariance} For every $f\in X_n^d$ and $C\in\GL_n$ the following holds:
$$
\Phi(Cf)=(\det C)^2\,\Bigl(C\Phi(f)\Bigr).
$$
In particular, the morphism $\Phi$ is $\SL_n$-equivariant. 
\end{proposition}

Finally, we note that the associated form of $f\in X_n^d$ arises from the following invariantly defined map
$$
\fm/\fm^2 \to \Soc(M_f),\quad x \mapsto y^{n(d-2)},\label{coordinatefree}
$$
with $y\in\fm$ being any element that projects to $x\in\fm/\fm^2$. Indeed, ${\mathbf f}$ is derived from this map by identifying the target with $\CC$ via $\omega$ and the source with $\CC^{n*}$ by mapping the image of $\widehat{z}_j$ in $\fm/\fm^2$ to $e_j^*$, $j=1,\dots,n$. It then follows that for any rational $\GL_n$-invariant function $R$ on ${\mathcal Q}_n^{n(d-2)*}$ the value $R({\mathbf f})$ depends only on the isomorphism class of the algebra $M_f$. As stated in the introduction, this is another fact that links Conjecture \ref{conj2} with the reconstruction problem. For other properties of the associated form and the morphism $\Phi$ the interested reader is referred to \cite[Section 2]{AI2} and \cite[Section 4]{F}. 
 
\section{The results}\label{S:contravariant}
\setcounter{equation}{0}

\subsection{Covariants and contravariants}
Recall that a polynomial $\Gamma \in\CC[\cQ_n^d\times\CC^n]$ is said to be a {\it covariant}\, of forms in $\cQ_n^d$ if for all $f\in\cQ_n^d$, $z\in\CC^n$ and $C\in\GL_n$ the following holds:
$$
\Gamma(f,z)=(\det C)^k\, \Gamma(C f,z\cdot C^t),
$$
where $k$ is an integer called the weight of\, $\Gamma$. Every homogeneous component of\, $\Gamma$ with respect to $z$ is automatically homogeneous with respect to $f$ and is also a covariant. Such covariants are called homogeneous and their degrees with respect to $f$ and $z$ are called the degree and order, respectively. We may consider a homogenous covariant $\Gamma$ of degree $K$ and order $D$ as the $\SL_n$-equivariant morphism
$$
\cQ_n^d  \to \cQ_n^D,\quad f  \mapsto (z \mapsto \Gamma(f,z))
 $$
 of degree  $K$ with respect to $f$, which maps a form $f\in\cQ_n^d$ to the form in $\cQ_n^D$ whose evaluation at $z$ is $\Gamma(f,z)$. Covariants independent of $z$ (i.e., of order $0$) are called {\it relative invariants}. Note that the discriminant $\Delta$ is a relative invariant of forms in $\cQ_n^d$ of weight $d(d-1)^{n-1}$ (see \cite[Chapter 13]{GKZ}).

Analogously, a polynomial $\Lambda \in\CC[\cQ_n^d\times \CC^{n*}]$ is said to be a {\it contravariant}\, of forms in $\cQ_n^d$  if for all $f\in\cQ_n^d$, $z^*=(z_1^*,\dots,z_n^*) \in \CC^{n*}$ and $C\in\GL_n$ one has
$$
\Lambda(f,z^*)=(\det C)^k\, \Lambda(C f,z^*\cdot C^{-1}),
$$
where $k$ is a (nonnegative) integer called the weight of\, $\Lambda$. Again, every contravariant splits into a sum of homogeneous ones, and for a homogeneous contravariant its degrees with respect to $f$ and $z^*$ are called the degree and class, respectively. 

We may consider a homogenous contravariant $\Lambda$ of degree $K$ and class $D$ as the $\SL_n$-equivariant morphism
$$
\cQ_n^d  \to \Sym^D(\CC^n),\quad  f  \mapsto (z^* \mapsto \Lambda(f,z^*))
$$
of degree $K$ with respect to $f$. Upon the standard identification $\Sym^D(\CC^n)=(\Sym^D (\CC^{n*}))^*=\cQ_n^{D*}$ induced by the polar pairing, this morphism can be regarded as a map from $\cQ_n^d$ to $\cQ_n^{D*}$. We will abuse notation by using the same symbol to denote both an element in $\CC[\cQ_n^d \times \CC^{n*}]$ and the corresponding morphism $\cQ_n^d \to \cQ_n^{D*}$.

\subsection{The contravariant arising from $\Phi$ and its degree}
Recall that the morphism $\Phi$ is a map
$$
\Phi \co X_n^d \to \cQ_n^{n(d-2)*}
$$
defined on the locus $X_n^d$ of nondegenerate forms.  The coefficients $\mu_{i_1, \ldots, i_n}$ that determine $\Phi$ (see (\ref{assocformexpppp}), (\ref{assocformexpp})) are elements of the coordinate ring $\CC[X_n^d] = \CC[\cQ_n^d]_{\Delta}$.  Let $p_{i_1,\dots,i_n}$ be the minimal integer such that $\Delta^{p_{i_1,\dots,i_n}}\cdot\mu_{i_1,\dots,i_n}$ is a regular function on ${\mathcal Q}_n^d$ (see formula (\ref{formulaformus})) and
$$
p:=\max\{p_{i_1,\dots,i_n}\mid i_1+\dots+i_n=n(d-2)\}.
$$
Then the product $\Delta^p \Phi$ is the morphism
$$
\Delta^p \Phi \co X_n^d \to \cQ_n^{n(d-2)*}, \quad f \mapsto \Delta(f)^p \Phi(f),
$$
which extends to a morphism from $\cQ_n^d$ to $\cQ_n^{n(d-2)*}$. We denote the extended map by the same symbol $\Delta^p \Phi$. 

Notice that, by Proposition \ref{equivariance}, the morphism 
$$
\Delta^p \Phi \co \cQ_n^d \to \cQ_n^{n(d-2)*}
$$
is in fact a homogeneous contravariant of weight $pd(d-1)^{n-1}-2$.  Since the class of $\Delta^p \Phi$ is $n(d-2)$, it follows that its degree is equal to $np(d-1)^{n-1}-n$. Observe that this last formula implies $p>0$ as the degree of a contravariant is always nonnegative.

We are now ready to state the first result of this paper.

\begin{theorem}\label{main} 
We have
\begin{equation}
p\le\left[\frac{n^{n-2}}{(n-1)!}\right],\label{estim}
\end{equation}
where $[x]$ denotes the largest integer that is less than or equal to $x$. Hence the degree of $\Delta^p \Phi$ does not exceed $n[n^{n-2}/(n-1)!](d-1)^{n-1}-n$.
\end{theorem}

\begin{proof} We start as in the proof of Proposition \ref{regular} by solving system (\ref{keyeq}), which was done by utilizing Cramer's rule for any form $f$ lying in a Zariski open subset $U$ of $X_n^d$. Let, as before, $A(f)$ be the matrix of this system and $D(f):=\det A(f)$. We see from (\ref{keyeq}) that the entries of the first $N-1$ columns of the matrix $A(f)$ are linear functions of the coefficients of the form $f$, whereas the entries of the $N$th column are homogeneous polynomials of degree $n$ in these coefficients, where $N:=\dim_{\CC}\cQ_n^{n(d-2)}$. Therefore, $D$ is a homogeneous polynomial of degree
\begin{equation}
\delta_1:=N+n-1\label{delta1}
\end{equation}
on $\cQ_n^d$ nonvanishing on $U$. It then follows that the degree of the minimal denominator of $\mu_{i_1,\dots,i_n}$ does not exceed $\delta_1$.

At the same time, $\Delta^{p_{i_1,\dots,i_n}}\cdot\mu_{i_1,\dots,i_n}$ is a regular function on ${\mathcal Q}_n^d$ (recall that $p_{i_1,\dots,i_n}$ is the minimal integer with this property). It is well-known that $\Delta$ is an irreducible homogeneous polynomial of degree 
\begin{equation}
\delta_2:=n(d-1)^{n-1}\label{delta2}
\end{equation}
on ${\mathcal Q}_n^d$ (the irreducibility of $\Delta$ can be observed by considering an incidence variety as in \cite[p.~169]{Mu}). Therefore, the degree of the minimal denominator of $\mu_{i_1,\dots,i_n}$ is\linebreak $p_{i_1,\dots,i_n}\cdot\delta_2$, hence one has
\begin{equation}
p_{i_1,\dots,i_n}\cdot\delta_2\le\delta_1.\label{eq6}
\end{equation}

We now need a lemma:

\begin{lemma}\label{estimate}
We have
\begin{equation}
\delta_1\le \frac{n^{n-2}}{(n-1)!}\delta_2.\label{eq1}
\end{equation} 
\end{lemma}

\begin{proof} From formulas (\ref{dimform}), (\ref{delta1}) we find
\begin{equation}
\delta_1=\left(
\begin{array}{c}
n(d-1)-1\\
\vspace{-0.1cm}\\
n-1
\end{array}
\right)+n-1.\label{eq2}
\end{equation}
On the other hand, elementary combinatorics together with (\ref{delta2}) yields
\begin{equation}
\left(
\begin{array}{c}
n(d-1)\\
\vspace{-0.1cm}\\
n-1
\end{array}
\right)\le\frac{(n(d-1))^{n-1}}{(n-1)!}= \frac{n^{n-2}}{(n-1)!}\delta_2\label{eq3}
\end{equation}
and
\begin{equation}
\left(
\begin{array}{c}
n(d-1)\\
\vspace{-0.1cm}\\
n-1
\end{array}
\right)=
\left(
\begin{array}{c}
n(d-1)-1\\
\vspace{-0.1cm}\\
n-1
\end{array}
\right)+
\left(
\begin{array}{c}
n(d-1)-1\\
\vspace{-0.1cm}\\
n-2
\end{array}
\right).\label{eq4}
\end{equation}
It is then clear that inequality (\ref{eq1}) follows from (\ref{eq2}), (\ref{eq3}), (\ref{eq4}) provided we show that
\begin{equation}
\left(
\begin{array}{c}
n(d-1)-1\\
\vspace{-0.1cm}\\
n-2
\end{array}
\right)\ge n-1.\label{eq5}
\end{equation}

For $n=2$ inequality (\ref{eq5}) is obvious, and we assume that $n\ge 3$. Then, as $n-2>0$, we have
$$
\left(
\begin{array}{c}
n(d-1)-1\\
\vspace{-0.1cm}\\
n-2
\end{array}
\right)\ge
n(d-1)-1\ge n-1, 
$$
which establishes (\ref{eq5}) and concludes the proof of the lemma.\end{proof}

The theorem now follows by combining inequalities (\ref{eq6}) and (\ref{eq1}).\end{proof}

Notice that the expression
$$
\frac{n^{n-2}}{(n-1)!}\delta_2=\frac{(n(d-1))^{n-1}}{(n-1)!}
$$
in the right-hand side of (\ref{eq1}) is simply the leading term in $\delta_1$ at infinity with respect to the parameter $(d-1)$ (see (\ref{eq2})). Thus, the content of Lemma \ref{estimate} is that the value of $\delta_1$ is controlled from above by its asymptotic behavior for {\it every}\, $d$. 

\subsection{Estimates for the cases $n=4,5$}

Observe that for $n=2,3$ upper bound (\ref{estim}) yields $p=1$, which is the result of \cite[Proposition 3.1]{AIK}. However, (\ref{estim}) is not sharp in general. Specifically, in the following two propositions we investigate the cases $n=4$, $n=5$ and find that for sufficiently small values of $d$ estimate (\ref{estim}) can be improved. 

Indeed, if $n=4$ inequality (\ref{estim}) yields $p\le 2$, whereas in fact the following holds:

\begin{proposition}\label{n=4}
For $n=4$ one has
$$
\begin{array}{ll}
p=1 & \hbox{if $3\le d\le 6$,}\\
\vspace{-0.1cm}\\
p\le 2 & \hbox{if $d\ge 7$.}
\end{array}
$$
\end{proposition}

\begin{proof} Let us determine the values of $d$ for which $\delta_1<2\delta_2$. By expanding formulas (\ref{delta2}) and (\ref{eq2}) for $n=4$, we see that this inequality is equivalent to
\begin{equation}
4d^3-36d^2+71d-36< 0.\label{eq7}
\end{equation}
Set $F(t):=4t^3-36t^2+71t-36$ and write the left-hand side of (\ref{eq7}) as 
$$
F(d)=d(4d^2-36d+71)-36.  
$$
The expression $4d^2-36d+71$ is easily seen to be negative for $3\le d\le 6$, thus $F(d)<0$ for these values of $d$ as well. On the other hand, one calculates $F(7)>0$ and observes that $F'(t)=12t^2-72t+71$ is positive for $t\ge 7$, hence $F(d)>0$ for all $d\ge 7$.

Thus, we have shown that (\ref{eq7}) holds precisely for $3\le d\le 6$. For these values of $d$ inequality (\ref{eq6}) yields $p=1$, which completes the proof. \end{proof}

\begin{remark}\label{exact1}
The proof of Proposition \ref{n=4} shows that the bound $p\le 2$ cannot be improved for any $d\ge 7$ if one relies just on (\ref{eq6}) and the inequality $\delta_1<2\delta_2$.   
\end{remark}

Next, we will investigate the case $n=5$, which is computationally more demanding. In this situation, inequality (\ref{estim}) yields $p\le 5$, but we will obtain the following more precise bounds:

\begin{proposition}\label{n=5}
For $n=5$ one has
$$
\begin{array}{ll}
p=1 & \hbox{if $d=3$,}\\
\vspace{-0.1cm}\\
p\le 2 & \hbox{if $d=4$,}\\
\vspace{-0.1cm}\\
p\le 3 & \hbox{if $5\le d\le 8$,}\\
\vspace{-0.1cm}\\
p\le 4 & \hbox{if $9\le d\le 50$,}\\
\vspace{-0.1cm}\\
p\le 5 & \hbox{if $d\ge 51$.}\\
\vspace{-0.1cm}\\
\end{array}
$$
\end{proposition}

\begin{proof} First, let us determine the values of $d$ for which $\delta_1<5\delta_2$. As in the proof of Proposition \ref{n=4}, by expanding formulas (\ref{delta2}) and (\ref{eq2}) for $n=5$, we see that this inequality is equivalent to
\begin{equation}
5d^4-270d^3+955d^2-1170d+504<0.\label{eq8}
\end{equation}
Set $F_1(t):=5t^4-270t^3+955t^2-1170t+504$ and write the left-hand side of (\ref{eq8}) as 
$$
F_1(d)=d^2(5d^2-270d+955)-1170d+504.  
$$
The expression $5d^2-270d+955$ is immediately seen to be negative for $4\le d\le 50$, thus, as $-1170d+504<0$ for all $d\ge 3$, it follows that $F_1(d)<0$ for $4\le d\le 50$. Also, one computes $F_1(3)<0$, and therefore we have $F_1(d)<0$ for all $3\le d\le 50$. On the other hand, one calculates $F_1(51)>0$, $F_1'(51)>0$ and observes that $F_1''(t)=60t^2-1620t^2+1910$ is positive for $t\ge 51$, hence $F_1(d)>0$ for all $d\ge 51$. Thus, we have proved that (\ref{eq8}) holds precisely for $3\le d\le 50$. For these values of $d$ inequality (\ref{eq6}) yields $p\le 4$.

Next, let us find the values of $d$ for which $\delta_1<4\delta_2$. As above, by expanding formulas (\ref{delta2}) and (\ref{eq2}) for $n=5$, we see that this inequality is equivalent to
\begin{equation}
29d^4-366d^3+1099d^2-1266d+528<0.\label{eq9}
\end{equation}
Set $F_2(t):=29t^4-366t^3+1099t^2-1266t+528$ and write the left-hand side of (\ref{eq9}) as 
$$
F_2(d)=d^2(29d^2-366d+1099)-1266d+528. 
$$
The expression $29d^2-366d+1099$ is negative for $5\le d\le 7$, thus $F_2(d)<0$ for these values of $d$ as well. Also, one computes $F_2(3)<0$, $F_2(4)<0$, $F_2(8)<0$ and therefore $F_2(d)<0$ for all $3\le d\le 8$. On the other hand, $F_2(9)>0$, $F_2'(9)>0$ and it is not hard to observe that $F_2''(t)=348t^2-2196t+2198$ is positive for $t\ge 9$, hence $F_2(d)>0$ for all $d\ge 9$. We have thus shown that (\ref{eq9}) holds precisely for $3\le d\le 8$. For these values of $d$ inequality (\ref{eq6}) leads to the bound $p\le 3$.

Further, let us determine the values of $d$ for which $\delta_1<3\delta_2$. For $n=5$ this inequality is equivalent to
\begin{equation}
53d^4-462d^3+1243d^2-1362d+552<0.    \label{eq10}
\end{equation}
Set $F_3(t):=53t^4-462t^3+1243t^2-1362t+552$ and observe that $F_3(3)<0$, $F_3(4)<0$. On the other hand, $F_3(5)>0$, $F_3'(5)>0$ and we see that $F_3''(t)=636t^2-2772t+2486$ is positive for $t\ge 5$, therefore $F_3(d)>0$ for all $d\ge 5$. Thus, (\ref{eq10}) holds precisely for $d=3,4$. For these values of $d$ inequality (\ref{eq6}) implies $p\le 2$.

Finally, let us find the values of $d$ for which $\delta_1<2\delta_2$. For $n=5$ this inequality is equivalent to
\begin{equation}
77d^4-558d^3+1387d^2-1458d+576<0.      \label{eq11}
\end{equation}
Set $F_4(t):=77t^4-558t^3+1387t^2-1458t+576$ and notice that $F_4(3)<0$. On the other hand, $F_4(4)>0$, $F_4'(4)>0$ and one observes that $F_4''(t)=924t^2-3348t+2774$ is positive for $t\ge 4$, hence $F_4(d)>0$ for all $d\ge 4$. Therefore, (\ref{eq11}) holds precisely for $d=3$. For this value of $d$ inequality (\ref{eq6}) yields $p=1$. 

The proof is now complete. \end{proof}

\begin{remark}\label{exact2}
As in remark \ref{exact1}, we note that the bounds obtained in Proposition \ref{n=5} cannot be improved if one relies just on (\ref{eq6}) and the inequalities $\delta_1<k\delta_2$ for $k=2,3,4,5$.   
\end{remark}

\begin{remark}\label{d=3}
One may get the impression that the method of proof of Propositions \ref{n=4}, \ref{n=5} always yields that $p=1$ if $d=3$.  In fact, this is not the case as the example of $n=6$ shows. Indeed, for $n=6$, $d=3$ the above approach only leads to the bound $p\le 2$.
\end{remark}

The analysis carried out in the proofs of Propositions \ref{n=4}, \ref{n=5} can be extended, in principle, to the case of any $n$. Indeed, both $\delta_1$ and $\delta_2$ may be viewed as polynomials of degree $n-1$ in $d$, and one may try to explicitly determine the values of $d$ for which the difference $k\delta_2-\delta_1$ is positive, where $k=2,\dots,[n^{n-2}/(n-1)!]$. However, an analysis of this kind appears to be computationally quite challenging to perform in full generality, and we did not attempt to do so.


\begin{thebibliography}{ABCD}

\bibitem[AI1]{AI1} Alper, J. and Isaev, A. V., Associated forms in classical invariant theory and their applications to
hypersurface singularities, {\it Math. Ann.} {\bf 360} (2014), 799--823.

\bibitem[AI2]{AI2} Alper, J. and Isaev, A. V., Associated forms and hypersurface singularities: the binary case, to appear in {\it J. reine angew. Math.}, published online, DOI: 10.1515/crelle-2016-0008.

\bibitem[AIK]{AIK} Alper, J., Isaev, A. V. and Kruzhilin, N. G., Associated forms of binary quartics and ternary cubics, {\it Transform. Groups} {\bf 21} (2016), 593--618.

\bibitem[Ba]{Ba} Bass, H., On the ubiquity of Gorenstein rings, {\it Math. Z.} {\bf 82} (1963), 8--28.

\bibitem[Be]{Be} Benson, M., Analytic equivalence of isolated hypersurface singularities defined by homogeneous polynomials, in: {\it Singularities {\rm (}Arcata, CA, 1981{\rm)}}, Proc. Sympos. Pure Math. 40, Amer. Math. Soc., Providence, RI, 1983, pp. 111--118.


\bibitem[D]{D} Dolgachev, I., {\it Classical Algebraic Geometry. A Modern View}, Cambridge University Press, Cambridge, 2012.
   

\bibitem[EI1]{EI1} Eastwood, M. G. and Isaev, A. V., Extracting invariants of isolated hypersurface singularities from their moduli algebras, {\it Math. Ann.} {\bf 356} (2013), 73--98.

\bibitem[EI2]{EI2} Eastwood, M. G. and Isaev, A. V., Invariants of Artinian Gorenstein algebras and isolated hypersurface singularities, in: {\it Developments and Retrospectives in Lie Theory: Algebraic Methods}, Developments in Mathematics 38, Springer, New York, 2014, pp. 159--173.

\bibitem[F]{F} Fedorchuk, M., GIT semistability of Hilbert points of Milnor algebras, to appear in {\it Math. Ann.}, published online, DOI: 10.1007/s00208-016-1377-2.




\bibitem[GKZ]{GKZ} Gelfand, I. M., Kapranov, M. M. and Zelevinsky, A. V., {\it Discriminants, Resultants and Multidimensional Determinants}, Modern Birkh\"auser Classics, Birkh\"auser Boston, Inc., Boston, MA, 2008.

\bibitem[GW]{GW} G\"ortz, U. and Wedhorn, T., {\it Algebraic Geometry I. Schemes with Examples and Exercises}, Advanced Lectures in Mathematics, Vieweg + Teubner, Wiesbaden, 2010.

\bibitem[GLS]{GLS} Greuel, G.-M., Lossen, C. and Shustin, E., {\it Introduction to Singularities and Deformations}, Springer Monographs in Mathematics, Springer, Berlin, 2007.

\bibitem[H]{H} Huneke, C., Hyman Bass and ubiquity: Gorenstein rings, in: {\it Algebra, $K$-theory, Groups, and Education {\rm (}New York, 1997{\rm )}}, Contemp. Math. 243, Amer. Math. Soc., Providence, RI, 1999, pp. 55--78.


\bibitem[I]{I} Isaev, A. V., On two methods for reconstructing homogeneous hypersurface singularities from their Milnor algebras, {\it Methods Appl. Anal.} {\bf21} (2014), 391--405.

\bibitem[IK]{IK} Isaev, A. V. and Kruzhilin, N. G., Explicit reconstruction of homogeneous isolated hypersurface singularities from their Milnor algebras, {\it Proc. Amer. Math. Soc.} {\bf 142} (2014), 581--590.



\bibitem[MY]{MY} Mather, J. and Yau, S. S.-T., Classification of isolated hypersurface singularities by their moduli algebras, {\it Invent. Math.} {\bf 69} (1982), 243--251.

\bibitem[Ma]{Ma} Matsumura, H., \textit{Commutative Ring Theory}, Cambridge Studies in Advanced Mathematics 8, Cambridge University Press, Cambridge, 1986.

\bibitem[Mu]{Mu} Mukai, S., {\it An Introduction to Invariants and Moduli}, Cambridge Studies in Advanced Mathematics 81, Cambridge University Press, Cambridge, 2003.


\bibitem[Sa]{Sa} Saito, K., Einfach-elliptische Singularit\"aten, {\it Invent. Math.} {\bf 23} (1974), 289--325.


\bibitem[Sh]{Sh} Shoshita\u\i shvili, A. N., Functions with isomorphic Jacobian ideals, {\it Funct. Anal. Appl.} {\bf 10} (1976), 128--133.



\end{thebibliography}
\end{document}